\documentclass{amsart}
\usepackage{graphicx,amssymb,color} 

\title{Optical Tomography with Scattered Rays}

\author[Chung]{Francis J. Chung}
\address{Department of Mathematics, University of Kentucky, Lexington KY 40506, USA}

\author[Hensley]{Faith Hensley}
\address{Department of Mathematics, University of Kentucky, Lexington KY 40506, USA}

\newtheorem{theorem}{Theorem}[section]

\newtheorem{proposition}[theorem]{Proposition}

\def \grad{\nabla}
\def \e{\varepsilon}

\usepackage{amsmath,amsthm,amssymb,color,latexsym}   
\usepackage{graphicx}
\usepackage{esint}
\usepackage{tikz}
\usetikzlibrary{intersections}

\newcommand\norm[1]{\left\lVert#1\right\rVert}

\numberwithin{equation}{section}

\begin{document}

\begin{abstract}
We consider the inverse problem of reconstructing the scattering coefficient of a simple radiative transport equation (RTE) used to model light propagation inside a scattering medium. To do so, we extract information from the second term in the collision expansion, that is, light that has been scattered by a single collision, for solutions to the RTE. We show that with proper sources and measurements, the scattering coefficient for the RTE can be obtained via an algebraic formula, resulting in a reconstruction with improved stability compared to the normal X-ray transform inversion method. We extend these theorems to apply to a multi-frequency setting in which photons change frequency after collisions. Then, we discuss potential applications of our theory for 3D image reconstruction. 
\end{abstract}

\maketitle

\section{Introduction}\label{Introduction} 

Optical tomography is the problem of reconstructing optical parameters on the interior of an object from measurements of light intensity on the boundary. If the medium to be imaged has low scattering, then this can be reduced to the problem of inverting the X-ray transform; this theory forms the mathematical basis for CT scans ~\cite{Qui, Lou}.  

In a regime with higher scattering, light propagation can be modeled using a radiative transport (or transfer) equation, abbreviated in either case by RTE ~\cite{DauLio}.  Then optical tomography becomes the mathematical inverse problem of reconstructing the coefficients of the RTE from boundary measurements of the solutions. This problem is known to be solvable; see for example ~\cite{ChoSte}. Many variations and extensions of this problem exist as well  -- see ~\cite{ArrSch} for a more extensive survey.

In these notes, we revisit the basic inverse problem of the RTE with a special focus on extracting information from light which has been scattered by a single collision.  The inverse problem for the RTE in this form has been previously well studied by a number of authors, usually under the name of broken-ray transforms or star transforms.  Approaches to this problem include Fourier transform ideas ~\cite{FloMarSchBRT, WalOSu} and cancellation and derivative based arguments ~\cite{KatKry, ZhaSchMar, FloMarSchNonR, KryKat}; similar ideas are also used to analyze single collision scattered light in ~\cite{Chu}. Numerical approaches have also been tried -- see ~\cite{GevSchCheGup, ItoTod} for some examples.  

Related concepts include conical Radon transforms and V-line transforms. See ~\cite{TerKucKun} and ~\cite{AmbSurvey}, respectively, for overviews of some of these ideas.

Our main result is that with correct inputs and measurements, the scattering coefficient for the RTE can be recovered via an algebraic formula. As a result of the simplicity of the formula, the reconstruction has very good stability compared to normal X-ray transform or derivative based inversions.  

To state the main results more precisely, let $\Omega$ be a bounded domain in $\mathbb{R}^n$, $n \geq 3$, with $C^1$ boundary. Suppose $\sigma \in C(\overline{\Omega})$ and $k \in C(\Omega \times S^{n-1} \times S^{n-1})$ are positive functions, and consider the RTE 
\begin{equation}\label{RTE} 
\theta \cdot \grad_x u(x,\theta) = -\sigma(x)u(x,\theta) + \int_{S^{n-1}}k(x,\theta,\theta') u(x,\theta') \, d\theta'. \end{equation}
Here $u : \Omega \times S^{n-1} \rightarrow \mathbb{R}$ is the specific intensity of light -- that is, $u(x,\theta)$ represents the intensity of light at $x$ in direction $\theta$. The coefficients $\sigma$ and $k$ can be thought of as governing absorption and scattering, respectively.  

Throughout these notes, we make the following two common assumptions on $k$ (See e.g. ~\cite{ChuSch,HenGre}): first, that for each fixed $x$ in $\Omega$, 
\begin{equation}\label{kSymmetry}
k(x,\theta,\theta') = k(x,\theta',\theta).
\end{equation}
(Note that in particular this holds if $k(x,\theta, \theta')$ depends only on $x$ and the angle between $\theta$ and $\theta'$, which is physically plausible.)

Second, for each fixed $x \in \Omega$, we assume that
\begin{equation}\label{Nongeneration}
\sigma(x) \geq \int_{S^{n-1}} k(x,\theta, \theta') \, d\theta;
\end{equation}
roughly speaking, this means that scattering does not generate new light.

To obtain a boundary value problem for this equation, we define 
\[
\partial \Omega_{\pm} = \{ (x,\theta) \in \partial \Omega \times S^{n-1}| \pm \nu(x) \cdot \theta > 0\},
\]
where $\nu(x)$ is the outward normal at $x \in \partial \Omega$. 

The following proposition guarantees the solvability of \eqref{RTE}.

\begin{proposition}\label{Solvability}
Suppose $f \in L^p(\partial \Omega_-)$, $1 \leq p \leq \infty$. Under the conditions on $\sigma$ and $k$ listed above, there exists a unique $u \in L^p(\Omega \times S^{n-1})$ such that $u$ solves 
\begin{align*}
\theta \cdot \grad_x u(x,\theta) &= -\sigma(x)u(x,\theta) + \int_{S^{n-1}}k(x,\theta,\theta') u(x,\theta') \, d\theta'. \\
u|_{\partial \Omega_-} &= f. \\
\end{align*}
\end{proposition}

For a proof, see \cite{EggSch} or the discussion in Section \ref{Collisions}.

For the purposes of these notes, we'll adopt a special notation for solutions to the boundary value problem in Proposition \ref{Solvability} with highly focused boundary conditions. 

To do this, we'll fix a function $\psi \in C^{\infty}_0(\mathbb{R}^n)$ with some normalization conditions (see the proof of Proposition \ref{OutputForms}). Given $\varepsilon > 0$ set 
\begin{equation}\label{ApproxIdentity}
\varphi_{-}(x'',\theta'') = \psi\left(\frac{x''-a}{\varepsilon}\right)  \psi\left(\frac{\theta''-\zeta}{\varepsilon}\right).
\end{equation}
Then for $(a,\zeta) \in \partial \Omega_-$, we can interpret $(a,\zeta)$ as a point in $\mathbb{R}^n \times \mathbb{R}^n$ by taking $S^{n-1}$ to be embedded in $\mathbb{R}^n$ as the standard unit sphere. Define $u^{\e}_{a,\zeta}$ and $U^{\e}_{a,\zeta}$ to be solutions to the RTE \eqref{RTE} with boundary conditions
\begin{equation}\label{DefnU}
u^{\e}_{a,\zeta}(x,\theta)|_{\partial \Omega_-} = \varphi_{-}(x-a,\theta - \zeta)  \mbox{ and } U^{\e}_{a,\zeta}(x,\theta)|_{\partial \Omega_-} = \e^{-n}\varphi_{-}(x-a,\theta - \zeta),
\end{equation}
respectively, where $\varphi_-$ acts as an approximation of identity on $\partial \Omega_-$.  

This brings us to the statement of the main theorem. 

\begin{theorem}\label{MainTheorem}
Let $x \in \Omega$, and $\zeta, \eta \in S^{n-1}$. There exist points $(a, \zeta), (b,\eta)$ on $\partial \Omega_-$ together with corresponding points $(c,\zeta)$ and $(d,\eta)$ on $\partial \Omega_+$, such that 
\begin{equation}
\lim_{\e \rightarrow 0} \frac{U^{\e}_{a,\zeta}(d,\eta)U^{\e}_{b,\eta}(c,\zeta) }{u^{\e}_{a,\zeta}(c,\zeta) u^{\e}_{b,\eta}(d,\eta)} = k^2(x, \zeta,\eta).
\end{equation}
\end{theorem}

Figure 1 illustrates the choice of $a,b,c,d$; see Section \ref{ProofofMains} for the full explanation.

This reconstruction is of interest for two major reasons.  One is that because it reconstructs the scattering kernel $k$ point by point, it raises the possibility of \emph{local reconstructions} even when global reconstructions are made difficult, e.g. by the presence of obstacles.

The second is that it follows directly from Theorem \ref{MainTheorem} that $k$ can be obtained from boundary data via an algebraic formula, which gives a remarkably stable reconstruction.  

The stability can be quantified via the albedo map $\mathcal{A}: L^\infty(\partial \Omega_-) \rightarrow L^\infty(\partial \Omega_+)$ by $\mathcal{A}(f) = u|_{\partial \Omega_+}$, where $u$ is the solution to the RTE \eqref{RTE} with boundary value $f$ on $\partial \Omega_-$.  

\begin{theorem}\label{Stability}
Suppose $\sigma_1, \sigma_2$ and $k_1, k_2$ are a priori bounded above by a constant $M$, and define albedo maps $\mathcal{A}_1$ and $\mathcal{A}_2$ respectively. Then there exists a constant $C$ depending only on $M$ and $\Omega$ such that 
\[
\|k_1^2 - k^2_2\|_{L^{\infty}(\Omega \times S^{n-1} \times S^{n-1})} \leq C\|\mathcal{A}_1 - \mathcal{A}_2\|_{\infty}.
\]
\end{theorem} 

Note that this is a Lipschitz-style stability result which does not depend on derivatives of the data, in comparison with the standard stability results for the inversion of the X-ray transform ~\cite{LouNat}.  

The remainder of these notes are organized as follows: In Section \ref{Collisions}, we gather some basic results about the RTE and its solutions.  The proofs of Theorem \ref{MainTheorem} and \ref{Stability} are given in Section \ref{ProofofMains}.  In Section \ref{MFT}, we consider versions of these theorems in the setting considered in ~\cite{FloMarSchNonR, KryKat} in which photons experience a change in frequency after collisions. Finally, in Section \ref{Conclusions}, we consider applications of the main theorems.  

\section{Collision Expansion}\label{Collisions}

First we need to examine the form of solutions to the RTE.  For this we'll need some notation.

Suppose $x, y \in \bar{\Omega}$, such that the line segment from $x$ to $y$ is contained in $\Omega$. We define the attenuation factor
\begin{equation}\label{Attenuation}
\alpha(x,y) = \exp\left(- \int_0^{\|y-x\|}\sigma(x + s(\widehat{y-x})) \, ds\right).
\end{equation}
Here $\widehat{y-x}$ is the unit vector in the direction of $y-x$. Note that $\alpha$ is symmetric in $x$ and $y$, and for $t \geq 0$ 
\begin{equation}\label{AttenuationDerivative}
\partial_t \alpha(x, x+ t\theta) = -\sigma(x + t\theta)\alpha(x, x+ t\theta).
\end{equation}
From this it follows that the solution to the transport equation
\begin{equation}\label{JEqn}
\begin{split}
\theta \cdot \grad u +\sigma u &= 0 \\
u|_{\partial \Omega_-} &= f \\
\end{split}
\end{equation}
is given by 
\begin{equation}\label{DefnJ}
u(x,\theta) = Jf(x,\theta) := \alpha(x, x_{\theta-})f(x_{\theta-}, \theta),
\end{equation}
where $x_{\theta \pm}$ is the first point on $\partial \Omega$ on the ray from $x$ in direction $\pm \theta$. 

Moreover, it follows that the solution to the transport equation
\begin{equation}\label{TEquation}
\begin{split}
\theta \cdot \grad u + \sigma u &= S(x,\theta)\\
u|_{\partial \Omega_-} &= 0 \\
\end{split}
\end{equation}
is given by 
\begin{equation}\label{TInverseDefn}
u(x,\theta) = T^{-1}S := \int_0^{\|x - x_{\theta-}\|} \alpha(x, x - t\theta)S(x - t\theta, \theta)dt.
\end{equation}
Note that since $\alpha < 1$ and $\|x - x_{\theta-}\|$ is uniformly bounded on $\Omega$, 
\begin{equation}\label{TInverseBound}
\|T^{-1} S\|_{L^{\infty}(\Omega \times S^{n-1})} \lesssim \|S\|_{L^{\infty}(\Omega \times S^{n-1})}.
\end{equation}

\begin{proposition}\label{CollisionExpansion}
Define the operator
\[
Ku(x,\theta) = \int_{S^{n-1}} k(x,\theta, \theta')u(x,\theta') d\theta'
\]
Then there exists $0 < C < 1$ such that
\begin{equation}\label{KBound}
\|T^{-1}K\|_{L^{\infty}(\Omega \times {S^{n-1}}) \rightarrow L^{\infty}(\Omega \times {S^{n-1}})} < C, 
\end{equation}
and the RTE \eqref{RTE} with boundary condition $u|_{\partial \Omega_-} = f$
for some $f \in L^{\infty}(\partial \Omega_{-})$, has a solution of the form $u$
\begin{equation}\label{CollisionExp}
u = (1 + T^{-1}K + (T^{-1}K)^2 + \ldots)Jf.
\end{equation}
\end{proposition}

\begin{proof}
To prove \eqref{KBound}, we substitute the definitions of $T^{-1}$ and $K$ to find that 
\[
T^{-1}Ku(x,\theta) = \int_0^{\|x-x_{\theta_-}\|}\alpha(x,x-t\theta)\int_{S^{n-1}}k(x-t\theta,\theta, \theta')u(x-t\theta, \theta')\, d\theta' \, dt.
\]
Keeping in mind that $\alpha$ and $k$ are positive, this gives us
\[
|T^{-1}Ku(x,\theta)| \leq \|u\|_{L^{\infty}(\Omega \times S^{n-1})}\int_0^{\|x-x_{\theta_-}\|}\alpha(x,x-t\theta)\int_{S^{n-1}}k(x-t\theta,\theta, \theta')\, d\theta' \, dt.
\]
Then condition \eqref{Nongeneration} implies that 
\[
|T^{-1}Ku(x,\theta)| \leq \|u\|_{L^{\infty}(\Omega \times S^{n-1})}\int_0^{\|x-x_{\theta_-}\|}\alpha(x,x-t\theta)\sigma(x-t\theta)\, dt,
\]
and now it follows from \eqref{AttenuationDerivative} that 
\begin{align*}
   |T^{-1}Ku(x,\theta)| &\leq \|u\|_{L^{\infty}(\Omega \times S^{n-1})}\int_0^{\|x-x_{\theta_-}\|}-\partial_t \alpha(x,x-t\theta)\, dt \\
   &= (1 - \alpha(x,x_{\theta_-}))\|u\|_{L^{\infty}(\Omega \times S^{n-1})}.\\
\end{align*}
Since $\alpha(x,x_{\theta_-})$ is bounded below by a constant depending only on $\Omega$ and $\sigma$, we get \eqref{KBound} as desired.  

For the remainder of the proof, it suffices to note that $u$ solves \eqref{RTE} with boundary condition $u|_{\partial \Omega_-} = f$ if and only if 
\[
u - T^{-1}Ku = Jf.
\]
Then the conclusion of the proof follows from \eqref{KBound} (see also the proofs in ~\cite{Chu, EggSch}).

\end{proof}


As noted in the introduction, we are in particular interested in solutions for the RTE with tightly focused boundary conditions designed using \eqref{ApproxIdentity}.  In this case we have the following propositions.


\begin{proposition}\label{OutputForms}
Let $(a,\zeta) \in \partial \Omega_-$ and define $u^{\e}_{a,\zeta}$ and $U^{\e}_{a,\zeta}$ to be solutions to the RTE \eqref{RTE} with boundary conditions defined in \eqref{DefnU}.  Then

i) if $c$ is the first point on $\partial \Omega$ reached from $a$ in direction $\zeta$, then 
\[
\lim_{\e \rightarrow 0} u_{a, \zeta}^{\e}(c,\zeta) = \alpha(a,c)
\]

ii) if $(d,\eta)$ is any other point on $\partial \Omega_-$, 
\[
\lim_{\e \rightarrow 0} U_{a, \zeta}^{\e}(d,\eta) = \alpha(a,x)k(x,\zeta,\eta) \alpha(x,d),
\] 
\indent \indent where $x$ is the intersection point of lines $a+s'\zeta$ and $d +s\eta$.

\end{proposition}

\begin{proof}[Proof of Proposition \ref{OutputForms}]

The proof follows from evaluating the integrals in the collision expansion, $u = (1 + T^{-1}K + (T^{-1}K)^2 + \ldots)Jf$. Moreover, we use the notation $u_i$ to denote the $i^{th}$ term of the collision expansion. That is, $u_1 = Jf$, $u_2 = T^{-1}KJf$, and $u_3 = ((T^{-1}K)^2 + \ldots)Jf$. As well, we denote $u_{a, \zeta}^{\e}(c,\zeta) = u(c, \zeta)$ and $U_{a, \zeta}^{\e}(d,\eta) = U(d, \eta)$ for ease of notation. In part i), one finds that all terms except the initial term go to zero in the limit. Similarly in part ii), all terms except the second term are zero in the limit.  See ~\cite{ChoSte, Chu, FloMarSchBRT} for similar propositions. 

For this proof, recall we mimic the setup of ~\cite{ChoSte} to construct the appropriate boundary condition, $\varphi_{-}$ concentrated at $(x', \theta')$. 

That is, we let $\psi \in C_{0}^{\infty}(\mathbb{R}^{n})$, $0 \leq \psi(x) \leq 1$, and $\psi(0)=1$. Given $\varepsilon>0$, set $$\varphi_{-}(x'', \theta'')= \psi\left(\frac{x''-a}{\varepsilon}\right) \psi\left(\frac{\theta''-\zeta}{\varepsilon}\right).$$ We then define $u$ and $U$ to be solutions to the RTE \eqref{RTE} with boundary conditions
\begin{equation}
u(x,\theta)|_{\partial \Omega_-} = \varphi_{-}(x-a,\theta - \zeta)  \mbox{ and } U(x,\theta)|_{\partial \Omega_-} = \e^{-n}\varphi_{-}(x-a,\theta - \zeta),
\end{equation}
respectively, where $\varphi_-$ acts as an approximation of identity on $\partial \Omega_-$. 

Let us expand first upon the proof of part i). Note we see from Theorem 4.1 in \cite{ChoSte} that all terms except the ballistic term $u_1$ go to zero in the limit. So, let us observe what happens to this ballistic term under the limit.
\begin{align*}
\lim_{\e \rightarrow 0} u(c,\zeta) &= \lim_{\e \rightarrow 0} u_1(c,\zeta)\\ &= \lim_{\e \rightarrow 0} J\varphi_{-}(c, \zeta) \\
&= \lim_{\e \rightarrow 0} \alpha(c,a) \varphi_{-}(a, \zeta) \text{ by definition of } J. \\
&= \alpha(c,a) \text{ since } \lim_{\e \rightarrow 0} \varphi_{-}(a, \zeta) =1 \text{ by our boundary condition.}\\
&= \alpha(a,c) \text{ since } \alpha \text{ is symmetric.}
\end{align*}

To expand on the proof of part ii), we will use Theorem 3.2 from \cite{ChoSte} and denote the second term of the collision expansion of the distribution kernel by  
\begin{align*} \tilde{U_{2}} = &\int_{0}^{\norm{d-d_{\eta_{-}}}} \alpha(d,d-t\eta) \, k(d-t\eta, \eta, \zeta) \\ & \times \fint_{S^{n-1}} \alpha(d-t\eta, (d-t\eta)_{\theta_{-}}) \, \delta_{\{x''\}} ((d-t\eta)_{\theta_{-}}) \, \delta_{\{\theta''\}}(\theta) \, d\theta \, dt.
\end{align*}
According to \cite{ChoSte}, we use the idea of a distribution kernel adapted to our notation to write the second term of our collision expansion as the integral over the incoming boundary of the distribution kernel against our boundary condition. Here, $U_2(d, \eta)$ denotes the second term of our collision expansion, $\tilde{U}_{2}(d, \eta, x'', \theta'')$ denotes the distribution kernel of the second term of our collision expansion, and $\varphi_{-}(x'', \theta'')$ denotes the boundary condition. Moreover, $d\xi(x'', \theta'')$ is the appropriate measure for this integral as seen in \cite{ChoSte} where  $d\xi(x'', \theta'') = d(\theta'') d\mu(x'')$ and $d\mu(x'')$ denotes the surface area of $\partial \Omega$. Now let us examine what happens to the single collision term.
$$U_{2}(d, \eta) = \int_{\partial\Omega_{-}} \tilde{U}_{2}(d, \eta, x'', \theta'') \varepsilon^{-n} \varphi_{-}(x'', \theta'') d\xi(x'', \theta'')$$

Using the definition of the distribution kernel and the boundary condition constructed above, we yield
\begin{align*}
U_{2}(d, \eta) &= \displaystyle \int_{\partial \Omega_{-}} \displaystyle \int_{0}^{\norm{d-d_{\eta_{-}}}} \alpha(d,d-t\eta) \, k(d-t\eta, \eta, \zeta) \fint_{S^{n-1}} \alpha(d-t\eta, (d-t\eta)_{\theta_{-}}) \\
&\hspace{0.8cm}\times \delta_{\{x''\}} ((d-t\eta)_{\theta_{-}}) \delta_{\{\theta''\}}(\theta) \, d\theta \, dt \, \psi \left(\frac{x''-a}{\varepsilon} \right) \psi\left(\frac{\theta''-\zeta}{\varepsilon} \right) \, \varepsilon^{-n} \, d\xi(x'', \theta'').
\end{align*}

Using the fact that $\int_{S^{n-1}} \delta_{\{\theta''\}} \, f(\theta) \, d\theta = f(\theta'')$, we yield
\begin{align*}
U_{2}(d, \eta) &= \displaystyle \int_{\partial \Omega_{-}} \displaystyle \int_{0}^{\norm{d-d_{\eta_{-}}}} \alpha(d,d-t\eta) \, k(d-t\eta, \eta, \zeta) \alpha(d-t\eta, (d-t\eta)_{\theta''_{-}}) \\
&\hspace{0.8cm}\times \delta_{\{x''\}} ((d-t\eta)_{\theta''_{-}}) \, dt \, \psi \left(\frac{x''-a}{\varepsilon} \right) \psi\left(\frac{\theta''-\zeta}{\varepsilon} \right) \, \varepsilon^{-n} \, d\xi(x'', \theta'')
\end{align*}

Breaking down our integral over the incoming boundary as an integral over the boundary and an integral over the unit sphere gives
\begin{align*}
U_{2}(d, \eta) &= \displaystyle \int_{\partial \Omega} \int_{S^{n-1}} \displaystyle \int_{0}^{\norm{d-_{\eta_{-}}}} \alpha(d,d-t\eta) \, k(d-t\eta, \eta, \zeta) \alpha(d-t\eta, (d-t\eta)_{\theta''_{-}}) \\
&\hspace{0.8cm}\times \delta_{\{x''\}} ((d-t\eta)_{\theta''_{-}}) \, dt \, \psi \left(\frac{x''-a}{\varepsilon} \right) \psi\left(\frac{\theta''-\zeta}{\varepsilon} \right) \, \varepsilon^{-n} \, d\theta'' \, d\mu(x'')
\end{align*}

Renormalizing $\psi$ if necessary and then using the fact that $$\lim_{\varepsilon \to 0} \int_{S^{n-1}} \varepsilon^{1-n} \, \psi \left( \frac{\theta'' - \zeta}{\varepsilon} \right) \, f(\theta'') \, d \theta'' = f(\zeta),$$ we yield
\begin{align*}
\lim_{\varepsilon \to 0} U_{2}(d, \eta) &= \lim_{\varepsilon \to 0} \displaystyle \int_{\partial \Omega} \displaystyle \int_{0}^{\norm{d-d_{\eta_{-}}}} \alpha(d,d-t\eta) \, k(d-t\eta, \eta, \zeta) \alpha(d-t\eta, (d-t\eta)_{\zeta_{-}}) \\
&\hspace{0.8cm}\times \delta_{\{x''\}} ((d-t\eta)_{\zeta_{-}}) \, dt \, \psi \left(\frac{x''-a}{\varepsilon} \right) \, \varepsilon^{-1} \, d\mu(x'')
\end{align*}

Using the fact that $\int \delta_{a}(b) \, f(a) \, da = f(b)$, we find that $\displaystyle \lim_{\varepsilon \to 0} U_{2}(d, \eta)$ can be rewritten as

$$\lim_{\varepsilon \to 0} \displaystyle \int_{0}^{\norm{d-d_{\eta_{-}}}} \alpha(d,d-t\eta) \, k(d-t\eta, \eta, \zeta) \alpha(d-t\eta, (d-t\eta)_{\zeta_{-}}) \, \psi \left(\frac{ (d-t\eta)_{\zeta_{-}}-a}{\varepsilon} \right) \, \varepsilon^{-1}  \, dt$$

Then, note $\psi$ is only supported when $(d-t\eta)_{\zeta_{-}}$ is close to $a$. In fact, again renormalizing $\psi$ as necessary, $$\lim_{\varepsilon \to 0} \int_{0}^{\norm{d-d_{\eta_{-}}}} \varepsilon^{-1} \, \psi \left( \frac{(d-t\eta)_{\zeta_{-}} - a}{\varepsilon} \right) \, f(t) \, dt = f(t_{0})$$ 

where $t_{0}$ is such that $(d-t_{0}\eta)_{\zeta_{-}} - a = 0$. So,

$$\lim_{\varepsilon \to 0} U_{2}(d, \eta) = \lim_{\varepsilon \to 0} \alpha(d, d-t_{0}\eta) \, k(d-t_{0}\eta, \eta, \zeta) \, \alpha(d-t_{0}\eta, a).$$

Let us note $(d-t_{0}\eta) = x$ as defined in the proposition which yields

$$\lim_{\varepsilon \to 0} U_{2}(d, \eta) = \lim_{\varepsilon \to 0} \alpha(d,x) \, k(x, \eta, \zeta) \, \alpha(x, a).$$


Moreover, note the other terms of our collision expansion aside from the single collision term go to zero when we take our limit. For our ballistic term $U_1$, $J \varphi_-$ is only supported along the line from $a$ to $c$ in the $\zeta$ direction. Our output $(d, \eta)$ is not on this line and thus outside of its support, meaning this ballistic term is zero in our limit. As well, our multiple collision term, $U_3$ also goes to zero in our limit. We can see this by following a similar technique as from \cite{ChoSte} noting $|\tilde{U_3}(d, \eta, x'', \theta'')| \in L^{\infty}(\partial \Omega_-; L^{1}(\partial \Omega_+, d \xi))$ where $\tilde{U_3}$ denotes the distribution kernel of $U_3$ defined such that $U_{3}(d, \eta) = \int_{\partial\Omega_{-}} \tilde{U}_{3}(d, \eta, x'', \theta'') \varphi_{-}(x'', \theta'') d\xi(x'', \theta'')$. Examine then that
\begin{align*}
    \norm{U_{3}(d, \eta)}_{L^1(\partial \Omega_+)} &\leq \int_{\partial \Omega_+} \left| \int_{\partial\Omega_{-}} \tilde{U}_{3}(d, \eta, x'', \theta'') \varphi_{-}(x'', \theta'') \varepsilon^{-n} d\xi(x'', \theta'') \right| d \xi (d, \eta) \\
    &\leq \int_{\partial \Omega_+} \int_{\partial\Omega_{-}} |\tilde{U}_{3}(d, \eta, x'', \theta'')| |\varphi_{-}(x'', \theta'')| \varepsilon^{-n} d\xi(x'', \theta'')  d \xi (d, \eta) \\
    &\leq \int_{\partial \Omega_-} \int_{\partial\Omega_{+}} |\tilde{U}_{3}(d, \eta, x'', \theta'')| d \xi (d, \eta) |\varphi_{-}(x'', \theta'')| \varepsilon^{-n} d\xi(x'', \theta'')  \\ 
    &\leq M \int_{\partial \Omega_-} |\varphi_{-}(x'', \theta'')| \varepsilon^{-n} d\xi(x'', \theta'') \\ 
    &\leq M' \varepsilon \text{ if } n\geq 3
\end{align*} 
Note to achieve the final line of the above calculations we write $$\varepsilon^{-n}\varphi_-(x'', \theta'') = \varepsilon^{-n} \psi \left(\frac{x''-a}{\varepsilon} \right) \psi\left( \frac{\theta''- \zeta}{\varepsilon} \right).$$ Note $\psi$ is compactly supported on the boundary of $\mathbb{R}^n$ and a function of $n-1$ variables. So, if we shrink by $\varepsilon$ and change variables we yield $\varepsilon^{-n} \varepsilon^{n-1} \varepsilon^{n-1} = \varepsilon^{n-2}$ which is bounded above by $\varepsilon$ provided $n\geq 3$. Then, by taking the limit as $\varepsilon \to 0$ we see that our multiple collision term, $u_3$ indeed goes to zero. This completes the proof.
\end{proof}

\section{Proof of Theorems \ref{MainTheorem} and \ref{Stability} }\label{ProofofMains}

We begin with the proof of Theorem \ref{MainTheorem}.

\begin{proof}[Proof of Theorem \ref{MainTheorem}]

Suppose $x \in \Omega$, and $\zeta,\eta \in S^{n-1}$. Let $a,b,c$ and $d$ be the first points on $\partial \Omega$ obtained by traveling from $x$ in the directions $-\zeta,-\eta,\zeta,$ and $\eta$, respectively, as illustrated in the diagram below.

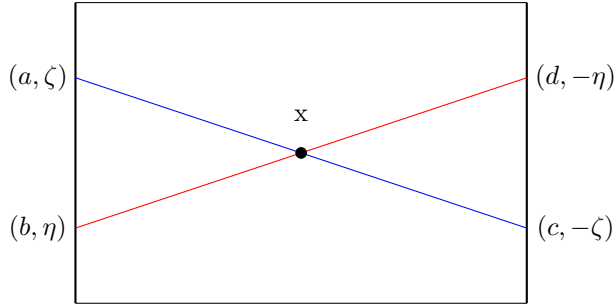
\begin{figure}[hbt!]
\begin{center}
\begin{tikzpicture}
    \coordinate (A) at (0,2);
    \coordinate (B) at (6,2);
    \coordinate (C) at (6,6);
    \coordinate (D) at (0,6);
    
    \draw[thick] (A) -- (D);
    \draw[thick] (B) -- (C);
    \draw (A) -- (B);
    \draw (C) -- (D);
    
    \draw[blue] (0,5) -- (3,4);
    \draw[blue] (3,4) -- (6,3);
    \draw[red] (0,3) -- (3,4);
    \draw[red] (6,5) -- (3,4);
    \filldraw[black] (3,4) circle (2 pt);

    \node at (-.5,5) {$(a, \zeta)$};
    \node at (6.65, 3) {$(c, -\zeta)$};
    \node at (-.5, 3) {$(b, \eta)$};
    \node at (6.65, 5) {$(d,-\eta)$};
    \node at (3, 4.5) {x};
    
\end{tikzpicture}
\end{center}
\caption{Ray Construction for Theorem 4.1}
    \label{fig:4.1Diagram}
\end{figure}

Now let $\varepsilon > 0$ and consider the solutions $u^{\varepsilon}_{a,\zeta}$, $u^{\varepsilon}_{b,\eta}$, $U^{\varepsilon}_{a,\zeta}$, and $U^{\varepsilon}_{b,\zeta}$, with boundary sources defined by \eqref{DefnU}.

By Proposition \ref{OutputForms}, we have 
\begin{align*}
\lim_{\varepsilon \rightarrow 0} u^{\varepsilon}_{a,\zeta}(c,\zeta) = \alpha(a,c) &= \alpha(a,x)\alpha(x,c) \\
\lim_{\varepsilon \rightarrow 0} u^{\varepsilon}_{b,\eta}(d,\eta) = \alpha(b,d) &= \alpha(b,x)\alpha(x,d) \\
\end{align*}
Moreover, by the same proposition, we also have
\begin{align*}
\lim_{\varepsilon \rightarrow 0} U^{\varepsilon}_{a,\zeta}(d,\eta) &= \alpha(a,x)k(x,\zeta,\eta)\alpha(x,d) \\
\lim_{\varepsilon \rightarrow 0} U^{\varepsilon}_{b,\eta}(c,\zeta) &= \alpha(b,x)k(x,\eta,\zeta)\alpha(x,c). \\
\end{align*}

Since \eqref{kSymmetry} implies that $k(x,\zeta,\eta) = k(x,\eta,\zeta)$, the above equations immediately imply that  
\begin{equation}
\lim_{\e \rightarrow 0} \frac{U^{\e}_{a,\zeta}(d,\eta)U^{\e}_{b,\eta}(c,\zeta) }{u^{\e}_{a,\zeta}(c,\zeta) u^{\e}_{b,\eta}(d,\eta)} = k^2(x, \zeta,\eta).
\end{equation}
as claimed.  

\end{proof}

The stability result now follows directly from the formula in Theorem \ref{MainTheorem}.

\begin{proof}[Proof of Theorem \ref{Stability}]

Fix $(x,\zeta, \eta)$ in $\Omega \times S^{n-1} \times S^{n-1}$.  Let $u^{\e}_{a,\zeta}$ and $U^{\e}_{a,\zeta}$ be the solutions to the RTE with coefficients $\sigma_1$ and $k_1$ and boundary conditions given by \eqref{DefnU}, and let $v^{\e}_{a,\zeta}$ and $V^{\e}_{a,\zeta}$ be the corresponding solutions with coefficients $\sigma_2$ and $k_2$. By Theorem \ref{MainTheorem},
\begin{align*}
& k^2_1(x,\zeta,\eta) - k^2_2(x,\zeta,\eta) = \lim_{\e \rightarrow 0} \frac{U^{\e}_{a,\zeta}(d,\eta)U^{\e}_{b,\eta}(c,\zeta) }{u^{\e}_{a,\zeta}(c,\zeta) u^{\e}_{b,\eta}(d,\eta)} - \frac{V^{\e}_{a,\zeta}(d,\eta)V^{\e}_{b,\eta}(c,\zeta) }{v^{\e}_{a,\zeta}(c,\zeta) v^{\e}_{b,\eta}(d,\eta)} \\ 
&= \lim_{\e \rightarrow 0} \frac{U^{\e}_{a,\zeta}(d,\eta)U^{\e}_{b,\eta}(c,\zeta)v^{\e}_{a,\zeta}(c,\zeta) v^{\e}_{b,\eta}(d,\eta)- V^{\e}_{a,\zeta}(d,\eta)V^{\e}_{b,\eta}(c,\zeta)u^{\e}_{a,\zeta}(c,\zeta) u^{\e}_{b,\eta}(d,\eta)}{u^{\e}_{a,\zeta}(c,\zeta) u^{\e}_{b,\eta}(d,\eta)v^{\e}_{a,\zeta}(c,\zeta) v^{\e}_{b,\eta}(d,\eta)}. \\ 
\end{align*}
Using an algebraic formula, we can rewrite $p_1q_1r_1s_1 - p_2q_2r_2s_2$ as 
\[
(p_1-p_2)(q_1r_1s_1) + (q_1-q_2)(p_2r_1s_1) + (r_1-r_2)(p_2q_2s_1) + (s_1-s_2)(p_2q_2r_2).
\]
Applying this fact to the numerator, we find that we can decompose the expression for $k^2_1(x,\zeta,\eta) - k^2_2(x,\zeta,\eta)$ into four terms, the first of which takes the form 
\[
\lim_{\e \rightarrow 0} \frac{(U^{\e}_{a,\zeta}(d,\eta)- V^{\e}_{a,\zeta}(d,\eta))U^{\e}_{b,\eta}(c,\zeta)}{u^{\e}_{a,\zeta}(c,\zeta) u^{\e}_{b,\eta}(d,\eta)} 
\]
after cancellation.  Note that 
\[
(U^{\e}_{a,\zeta}(d,\eta)- V^{\e}_{a,\zeta}(d,\eta)) = (\mathcal{A}_1-\mathcal{A}_2)\varepsilon^{-n}(\varphi_{-}(\cdot - a, \cdot - \zeta))
\]
evaluated at $(d,\eta)$, and consequently
\[
|U^{\e}_{a,\zeta}(d,\eta)- V^{\e}_{a,\zeta}(d,\eta)| \leq |\mathcal{A}_1 - \mathcal{A}_2|_{\infty}.
\]
Meanwhile, it follows from the formulas given in Proposition \ref{OutputForms} that $\lim_{\e \rightarrow 0} U^{\e}_{b,\eta}(c,\zeta)$ is bounded above by $M$, and $u^{\e}_{a,\zeta}(c,\zeta)$ and $u^{\e}_{b,\eta}(d,\eta)$ are both bounded below by a quantity depending only on $M$ and the diameter of $\Omega$.  

The other three terms are bounded similarly, and the desired result follows.  
\end{proof}

\section{Multi-Frequency Theorems}\label{MFT}

In this section we consider an alternative model for light propagation in which photons change frequency after collisions.  In this setting, we consider a scenario with an initial light source of frequency $e$ is introduced at the boundary, and photons, upon being scattered once, adopt a new lower frequency $f$.  (Additional scattering events would change the frequency yet again, but as we will see, we will not measure at these frequencies, and can disregard the analysis of these terms). Following ~\cite{FloMarSchNonR}, we consider the system of equations 
\begin{equation}\label{MultiFreqRTE}
\begin{split}
\theta \cdot \grad_x u(x,\theta) &= -\sigma_e(x)u(x,\theta) \\
\theta \cdot \grad_x v(x,\theta) &= -\sigma_f(x)v(x,\theta) + \int_{S^{n-1}}k(x,\theta,\theta') u(x,\theta') \, d\theta'. \\\
\end{split}
\end{equation}
Here $u$ represents the specific intensity of the light at the original frequency $e$, with boundary source $\phi$,  and $v$ represents the specific intensity of the light at the post-collision frequency $f$.  As in ~\cite{FloMarSchNonR}, we can also understand this as a model of fluoresence, where $e$ is the excitation frequency and $f$ is the fluoresence frequency.  We assume that $\sigma_e$, $\sigma_f$, and $k$ are subject to the same restrictions as in Section \ref{Introduction}, 
with the additional condition that 
\begin{equation}\label{KAngle}
k(x,\theta_1,\theta_2) = k(x, \theta_1',\theta_2') \mbox{ whenever } \theta_1 \cdot \theta_2 = \theta_1' \cdot \theta_2'
\end{equation}
which is to say the dependence of $k(\cdot ,\theta_1,\theta_2)$ on $\theta_1$ and $\theta_2$ comes only from the angle between $\theta_1$ and $\theta_2$.

Following the analysis in Section \ref{Collisions}, we see that \eqref{MultiFreqRTE}, equipped with the boundary conditions 
\begin{equation}\label{MultiFreqBC}
(u,v)|_{\partial \Omega_-} = (\varphi,\psi), 
\end{equation}
has the solutions 
\begin{equation}\label{MFSolutions}
u = J_e\varphi \quad \mbox{ and } \quad  v = T^{-1}_fKu + J_f \psi = T^{-1}_fKJ_e \varphi + J_f \psi,
\end{equation}
where $J_e, J_f$ and $T^{-1}_f$ represent the operators $J$ and $T^{-1}$ as defined in Section \ref{Collisions} with absorption coefficients $\sigma_e$ and $\sigma_f$.  In particular, \eqref{MultiFreqRTE} equipped with boundary conditions \eqref{MultiFreqBC} has unique solutions $u,v \in L^p(\Omega \times S^{n-1})$ for each pair $\varphi, \psi \in L^p(\partial \Omega_-)$.

This brings us to the main theorems for this section.

\begin{theorem}\label{MFTheorem}
Let $x \in \Omega$. There exist eight sets of boundary conditions $(\varphi^{\e}_j, \psi^{\e}_j) \in L^p(\partial \Omega_-^2)$ $1 \leq j \leq 8$, together with four points $(a,-\zeta)$ and $(b,-\eta)$, $(c,\zeta)$ and $(d,\eta)$ on $\partial \Omega_+$, such that 
\begin{equation}\label{MFFormula}
\lim_{\e \rightarrow 0} \frac{v^{\e}_1(d,\eta)v^{\e}_2(c,\zeta)v^{\e}_3(b,-\eta)v^{\e}_4(a,-\zeta)}{u^{\e}_5(c,\zeta) u^{\e}_6(d,\eta)v^{\e}_7(c,\zeta)v^{\e}_8(d,\eta)} = k^4(x, \zeta,\eta).
\end{equation}
\end{theorem}

For an explanation of how the boundary conditions are chosen, see the proof below. Note that all boundary conditions can be thought of as being concentrated near the four boundary points $(a,\zeta),(b,\eta), (c,-\zeta),$ and $(d,-\eta)$. 

To state the corresponding stability result, define the albedo map $\mathcal{A}^{F}: L^{\infty}(\partial \Omega_-^2) \rightarrow L^{\infty}(\partial \Omega_+^2)$ by $\mathcal{A}^{F}(\varphi,\psi) = (u,v)|_{\partial \Omega_+}$, where $(u,v)$ is the solution to \eqref{MultiFreqRTE} with boundary value $(\varphi,\psi)$ on $\partial \Omega_-$.   

\begin{theorem}\label{MFStability}
Suppose $\sigma_{e,1}, \sigma_{f,1}, \sigma_{e,2}, \sigma_{f,2}$ and $k_1, k_2$ are a priori bounded above by a constant $M$, and define albedo maps $\mathcal{A}^{F}_1$ and $\mathcal{A}^F_2$ respectively.  Then there exists a constant $C$ depending only on $M$ and $\Omega$ such that 
\[
\|k_1^4 - k^4_2\|_{L^{\infty}(\Omega \times S^{n-1} \times S^{n-1})} \leq C\|\mathcal{A}^F_1 - \mathcal{A}^F_2\|_{\infty}.
\]
\end{theorem}

The proof of Theorem \ref{MFStability} follows from the formula in Theorem \ref{MFTheorem} in the same way that Theorem \ref{Stability} follows from the formula in Theorem \ref{MainTheorem}, so it remains only to prove Theorem \ref{MFTheorem}.  

\begin{proof}[Proof of Theorem \ref{MFTheorem}]

Suppose $x \in \Omega$, and $\zeta,\eta \in S^{n-1}$. Define $a,b,c$ and $d$ as in the proof of Theorem \ref{MainTheorem}. For $j = 1, \ldots, 8$ we define $(u^{\e}_j, v^{\e}_j)$ to be the solution to \eqref{MultiFreqRTE} with boundary conditions 
\begin{align*}
(u^{\e}_1, v^{\e}_1) &= (\e^{-n}\varphi_{-}(\cdot -a,\cdot - \zeta), 0) \\
(u^{\e}_2, v^{\e}_2) &= (\e^{-n}\varphi_{-}(\cdot -b,\cdot - \eta), 0) \\
(u^{\e}_3, v^{\e}_3) &= (\e^{-n}\varphi_{-}(\cdot -c,\cdot + \zeta), 0) \\
(u^{\e}_4, v^{\e}_4) &= (\e^{-n}\varphi_{-}(\cdot -d,\cdot + \eta), 0) \\
(u^{\e}_5, v^{\e}_5) &= (\varphi_{-}(\cdot -a,\cdot - \zeta), 0) \\
(u^{\e}_6, v^{\e}_6) &= (\varphi_{-}(\cdot -b,\cdot - \eta), 0) \\
(u^{\e}_7, v^{\e}_7) &= (0, \varphi_{-}(\cdot -a,\cdot - \zeta)) \\
(u^{\e}_8, v^{\e}_8) &= (0, \varphi_{-}(\cdot -b,\cdot - \eta)). \\
\end{align*}

Applying the discussion in Section \ref{Collisions} to \eqref{MFSolutions}, we find that 
\begin{align*}
\lim_{\e \rightarrow 0}v^{\e}_1(d,\eta) &= \alpha_e(a,x)k(x,\zeta,\eta)\alpha_f(x,d)  \\
\lim_{\e \rightarrow 0}v^{\e}_2(c,\zeta) &= \alpha_e(b,x)k(x,\eta,\zeta)\alpha_f(x,c) \\
\lim_{\e \rightarrow 0}v^{\e}_3(b,-\eta) &= \alpha_e(c,x)k(x,-\zeta,-\eta)\alpha_f(x,b)\\
\lim_{\e \rightarrow 0}v^{\e}_4(a,-\zeta) &= \alpha_e(d,x)k(x,-\eta,-\zeta)\alpha_f(x,a) \\
\end{align*}
and
\begin{align*}
\lim_{\e \rightarrow 0}u^{\e}_5(c,\zeta) = \alpha_e(a,c) &= \alpha_e(a,x)\alpha_e(x,c)  \\
\lim_{\e \rightarrow 0}u^{\e}_6(d,\eta) = \alpha_e(b,d) &= \alpha_e(b,x)\alpha_e(x,d) \\
\lim_{\e \rightarrow 0}v^{\e}_7(c,\zeta) = \alpha_f(a,c) &= \alpha_f(a,x)\alpha_f(x,c) \\
\lim_{\e \rightarrow 0}v^{\e}_8(d,\eta) = \alpha_f(b,d) &= \alpha_f(b,x)\alpha_f(x,d),\\
\end{align*}
where $\alpha_e$ and $\alpha_f$ represent the attenuation factors for $\sigma_e$ and $\sigma_f$ respectively.  

Combining these formulas and using \eqref{KAngle} gives \eqref{MFFormula}, which finishes the proof.

\end{proof}

\section{Applications and Conclusions}\label{Conclusions} 

As discussed in the introduction, the inversion of the X-ray transform forms the mathematical basis for modern day CT scans, in which actual X-rays are sent through a patient or object to be imaged, and the boundary measurements are used to reconstruct a 3D image of the optical parameters on the inside.  

The main results of these notes suggest alternative procedures for obtaining a 3D image of optical parameters, the details of which depend on whether we consider the models of Theorem \ref{MainTheorem}, where post-collision photons retain the same frequency as before, or the model of Theorem \ref{MFTheorem}, where post-collision photons change frequency.  

In the single frequency model, the reconstruction of Theorem \ref{MainTheorem} requires only two X-ray beams to pass through a point $x$ in the region to be imaged in order to reconstruct an optical parameter at $x$.  

Therefore, for any two fixed directions $\zeta$ and $\eta$, any arrangement of sources and detectors that allows for each point $x$ to be illuminated and viewed from these directions, as in Figure 1, suffices to reconstruct a 3D image of optical parameters.  

For example, consider the diagram below in which narrow beam X-ray sources angled in directions $\zeta$ and $\eta$ are moved over the planes $A$ and $B$, respectively, with corresponding detectors along $C$ and $D$, enclosing the region to be imaged on the interior.  

\begin{center}
\includegraphics[scale=0.3]{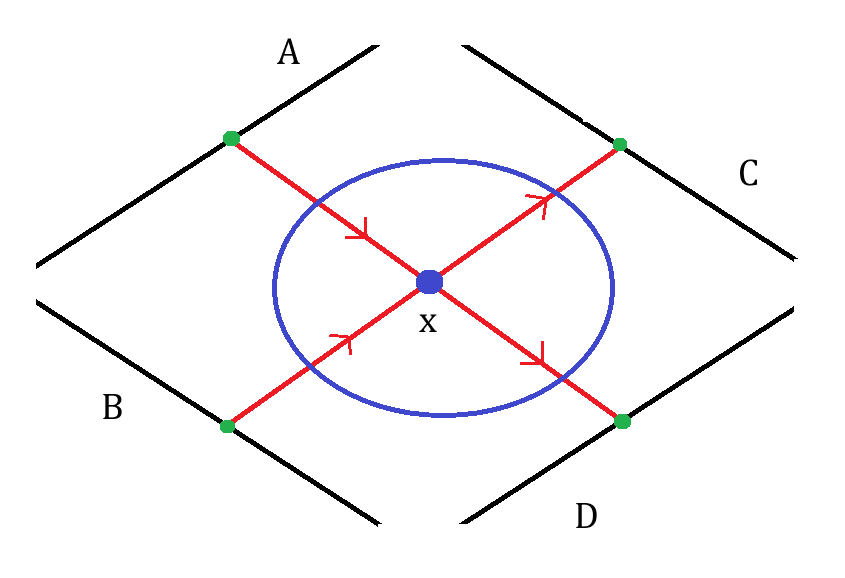}

{\tiny Figure 2: Sources are moved across $A$ and $B$ at fixed angles, and resulting beams are measured at the corresponding points on $C$ and $D$.}
\end{center}

Up to the necessary discrete approximation, each point $x$ in the region to be imaged is then the target of two X-ray beams emitted from $A$ and $B$, in the directions $\zeta$ and $\eta$, and can be viewed from corresponding detectors on $C$ and $D$, as depicted in the diagram.

This sets up the situation in Figure 1, so by applying sources and measurements described in Theorem \ref{MainTheorem} and applying the formula there gives $k^2(x,\zeta, \eta)$. Repeating for each $x$ in the domain gives a 3D reconstruction of $k^2(x,\zeta,\eta)$ as a function of $x$.  

As mentioned in the discussion of Theorem \ref{Stability}, the mathematical stability of the reconstruction described by Theorem \ref{MainTheorem} is substantially better than that of the standard X-ray transform, potentially leading to improved reconstructions with less data. In addition, the fact that reconstructions are done point by point raises the possibility of local reconstruction of areas of interest, even if obstructions (with unreasonably high absorption, say) prevent global reconstruction.

Other arrangements are also possible, especially if additional assumptions are made on the form of $k$.  For example, if $k(x,\zeta, \eta)$ takes the form 
\[
k(x,\zeta,\eta) = s(x)\kappa(\zeta,\eta)
\]
for some known $\kappa(\zeta,\eta)$ (e.g., a Henyey-Greenstein phase function ~\cite{HenGre}), then for any choice of $\zeta$, $\eta$, the reconstruction of $k(x,\zeta,\eta)$ can be divided pointwise by $\kappa(\zeta,\eta)$ to obtain a reconstruction of $s(x)$.  This leaves open the possibility of other geometric arrangements of sources and detectors. For example, if a horizontal plane is illuminated with a fan of X-ray beams, and imaged from a point above the plane, scattering measurements for the whole plane can be made simultaneously, albeit at different scattering angles. Supplementing with appropriate measurements of the ballistic rays, one obtain the measurements needed to apply the formula in Theorem \ref{MainTheorem}.  

Similar setups can also be considered for the multifrequency model, albeit with additional complications due to the extra boundary sources required by Theorem \ref{MFTheorem}.  As in the single frequency case, Theorem \ref{MFStability} guarantees improved mathematical stability in the reconstruction, giving the potential for improved reconstructions with less data.

\end{document}